\documentclass[12pt,a4paper]{article}
\usepackage[utf8]{inputenc}
\usepackage[T1]{fontenc}

\usepackage[printsolution=true]{exercises}

\usepackage{textcomp}
\usepackage{graphicx}
\usepackage{amsfonts}
\usepackage{latexsym}

\usepackage{amsmath,amsfonts,amssymb,amscd,amsthm}

\usepackage{subfigure,url}

\usepackage{color}
\usepackage{tabu}

\usepackage{tikz}

\usepackage{mathtools}
\usepackage{enumitem}
\usepackage{pifont}

\newtheorem{problem}{Problem}
\newtheorem{theo}[problem]{Theorem}

\newtheorem{defin}[problem]{Definition}
\newtheorem{prop}[problem]{Proposition}
\newtheorem{cor}[problem]{Corollary}
\newtheorem{lema}[problem]{Lemma}

\usepackage{bbm}

\begin{document}
\date{July 31, 2021}
\title{Bier spheres of extremal volume \\ and generalized permutohedra}
\author{{Filip D. Jevti\'{c}} \\{\small Mathematical Institute}\\[-2mm] {\small SASA,   Belgrade}
\and Rade T. \v Zivaljevi\' c\\ {\small Mathematical Institute}\\[-2mm] {\small SASA,   Belgrade}}

\maketitle

\begin{abstract}{A Bier sphere $Bier(K) = K\ast_\Delta K^\circ$, defined as the deleted join of a simplicial complex and its Alexander dual $K^\circ$, is a purely combinatorial object (abstract simplicial complex). Here we study a hidden geometry of Bier spheres by describing their natural geometric realizations, compute their volume, describe an effective criterion for their polytopality, and associate to $K$ a natural fan $Fan(K)$,  related to the Braid fan. Along the way we establish a connection of Bier spheres  of maximal      volume with recent generalizations of the classical Van Kampen-Flores theorem and clarify the role of Bier spheres in the theory of generalized  permutohedra.}

\end{abstract}

One of the main new results of \cite{jevtic_bier_2019} was the observation (\cite[Theorem 3.1]{jevtic_bier_2019}) that each Bier sphere $Bier(K)$, defined as a canonical  triangulation of a $(n-2)$ sphere $S^{n-2}$ associated to an abstract simplicial complex $K\subsetneq 2^{[n]}$, admits a starshaped embedding in $\mathbb{R}^{n-1}$.

\medskip
It turns out that the radial fan $Fan(K)$ of the starshaped embedding of the Bier sphere $Bier(K)$, described in the proof of this result, is a coarsening of the \emph{Braid arrangement fan}. This fact was not emphasized in \cite{jevtic_bier_2019}, however it is interesting in itself and certainly deserves further study.

\medskip
Recall that the Braid arrangement fan is the normal fan of the standard permutohedron   \cite{ziegler_lectures_1995} and that the coarsening of the Braid fan leads to an important and well studied class of \emph{generalized permutohedra} \cite{postnikov_permutohedra_2009, postnikov_faces_2007, carr_coxeter_2006, devadoss_space_2002, zelevinsky_nested_2006}
or deformed permutohedra, as they are called by some authors.

\medskip
In this paper we take a closer look at the fan $Fan(K)$ (the \emph{canonical} or \emph{Bier fan} of a simplicial complex $K$), with the goal to clarify the role of Bier spheres in the theory of generalized permutohedra and study other  geometric properties of Bier spheres arising from this construction.

\bigskip
The main new results of the paper are the following.

In Section \ref{sec:Bier_fans} we give a combinatorial proof that $Fan(K)$ is refined by the braid fan, relying on the \emph{preposets-braid cones} dictionary from \cite{postnikov_faces_2007}. In particular we show that the maximal cones of $Fan(K)$ are associated with \emph{tree posets} which have precisely one node which is not a leaf.

In Section \ref{sec:volume} we study  Bier spheres (or rather the associated starshaped sets $Star(K)$) of extremal volume. In particular we show (Proposition \ref{prop:max-vol})  that Bier spheres of maximal volume are closely related to the class of \emph{nearly neighborly Bier spheres}, studied in \cite{bjorner_bier_2004}, and \emph{balanced simplicial complexes} \cite{jojic_tverberg_2021}, which provide a natural class of examples extending the classical Van Kampen-Flores theorem, see \cite[Theorem 3.5]{jojic_tverberg_2021}.

One of the consequences of Propositions \ref{prop:vol-3-cases} and \ref{prop:max-vol} is that all starshaped sets $Star(K)$ of maximal volume coincide with one and the same,  universal $(n-1)$-dimensional convex set (convex polytope), denoted by $\Omega_n$ and referred to as the \emph{Van Kampen-Flores polytope}. The structure of the Van Kampen-Flores polytope is clarified (and its name explained) in Sections \ref{sec:volume} and \ref{sec:hypersimplex}, in particular we show (Theorem \ref{thm:R_n=hypersimplex}) that the polar dual of $\Omega_n$ is affine-isomorphic to a \emph{median hypersimplex}.

In Section \ref{sec:wall-crossing} we prove a $K$-submodularity theorem which for polytopal Bier spheres plays the role similar to the role of classical submodular functions (polymatroids) in the theory of generalized permutohedra.
With the aid of this result we obtain a useful criterion for a Bier sphere to be polytopal.

\medskip
For the reader's convenience here is a glossary with brief descriptions of the main objects studied in this paper.

\medskip
$Bier(K) = K\ast_\Delta K^\circ$, the Bier sphere of $K$, is a combinatorial object (simplicial complex), defined as a deleted join of two simplicial complexes ($K$ and its Alexander dual $K^\circ$).

\medskip
$Fan(K) = BierFan(K)$, the \emph{canonical} or the \emph{Bier fan} of $K$, is a complete, simplicial fan in $H_0 \cong \mathbb{R}^{n-1}$, associated to a simplicial complex $K\subsetneq 2^{[n]}$.

\medskip
$\mathcal{R}_{\pm\delta}(Bier(K))$ is the \emph{canonical starshaped realization} of $Bier(K)$ described in \cite[Theorem 3.1]{jevtic_bier_2019}.

\medskip
$Star(K)$ is the starshaped body in $H_0$ whose boundary is the sphere $\mathcal{R}_{\pm\delta}(Bier(K))$.

\medskip
$\Omega_n$ is a universal, $(n-1)$-dimensional convex polytope (the Van Kampen-Flores polytope) which is equal, as a convex body, to $Star(K)$ for each Bier sphere of maximal volume.

\begin{figure}[htb]
    \centering
    \includegraphics{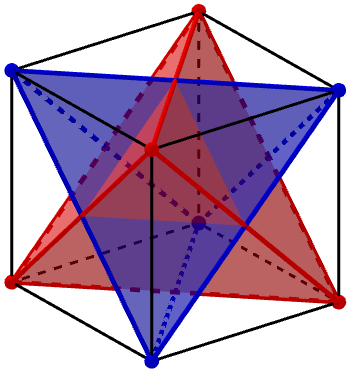}
    \caption{The $3$-dimensional cube as the Van Kampen-Flores polytope $\Omega_4$.}
    \label{fig:prva}
\end{figure}



\section{Bier fans of simplicial complexes}
\label{sec:Bier_fans}

Let $K\subsetneq 2^{[n]}$ be a simplicial complex and $K^\circ$ its Alexander dual.
By definition, see \cite{matousek_using_2008},  the associated Bier sphere is the deleted join,
\begin{equation}\label{eq:Bier-def}
Bier(K):= K \ast_\Delta K^\circ  .
\end{equation}
Side by side with the more standard $(A_1, A_2)\in Bier(K)$, we use an extended $\tau = (A_1,A_2; B)$ notation for simplices in the Bier sphere, where $B := [n]\setminus (A_1\cup A_2)$. Hence, an ordered partition $A_1\sqcup A_2 \sqcup B = [n]$ corresponds to a simplex $\tau\in Bier(K)$ if and only if $\emptyset \neq B \neq [n]$. In the ``interval notation'', used in \cite{bjorner_bier_2004}, the simplex $\tau = (A_1, A_2) = (A_1,A_2; B)$ can be recorded as the pair $(A_1, A_2^c)$.

For example the facets of $Bier(K)$ are triples $\tau = (A_1, A_2; B)$ where $B = \{\nu\}$ is a singleton. In this case $\tau$ is
(in the interval notation) determined by the pair $A\subsetneq C$, where $A = A_1\in K$ and $C = A_1\cup \{\nu\}\notin K$.

\medskip\noindent
The braid arrangement is the arrangement of hyperplanes $Braid_n = \{H_{i,j}\}_{1\leqslant i < j\leqslant n}$ in $H_0$ where
$H_0 := \{x\in \mathbb{R}^n \mid x_1+\dots+ x_n =0\} \cong \mathbb{R}^n/(1,\dots, 1)\mathbb{R}$ and
$H_{i,j} := \{ x \mid x_i - x_j =0\}$. The hyperplanes $H_{i,j}$ subdivide the space $H_0$ into the polyhedral
cones
\begin{equation}
C_\pi := \{x\in H_0 \mid x_{\pi(1)}\leqslant x_{\pi(2)}\leqslant \dots \leqslant x_{\pi(n)}\}
\end{equation}
labeled by permutations $\pi\in S_n$. The cones $C_\pi$, together with their faces, form a complete simplicial fan
in $H_0$, called the {\em braid arrangement fan}.

\subsection{Preposets and Bier fans}\label{sec:preposet}

 A binary relation $R\subseteq [n]\times [n]$ is a \emph{preposet} on $[n]$ if it is both \emph{reflexive} and \emph{transitive}. Following \cite{postnikov_faces_2007}, in explicit calculations we often write $\preccurlyeq_R$, instead of $R$, and $x\preccurlyeq_R y$, instead of $(x,y)\in R$ or $xRy$. Given a preposet $\preccurlyeq_R$ we write $x\prec_R y$ if $\preccurlyeq_R$ and $x\neq y$, and $x \equiv_R y$ if both $x \preccurlyeq_R y$ and $y \preccurlyeq_R x$.

\medskip
For a more detailed account and, in particular, the  \emph{preposet-braid cone dictionary}, which describes the geometry of braid cones in the language of preposets, the reader is refereed to \cite{postnikov_faces_2007} (Sections~3.3 and 3.4).

\medskip
Let $\tau = (A_1, A_2; B)\in Bier(K)$. The associated preposet $\preccurlyeq_\tau$ is the binary relation defined as
the reflexive closure of the relation
\begin{equation}\label{eqn:preposet}
\rho_\tau  := (A_1 \times B) \cup (B\times B) \cup (B\times A_2) \subseteq [n]\times [n] \, .
\end{equation}
Following \cite{postnikov_faces_2007} (Section~3.4), the associated {\em braid cone} is
\begin{equation}\label{eqn:braid-cone}
 Cone(\preccurlyeq_\tau) =  Cone(\tau) = Cone(A_1, A_2; B) = \{x\in H_0  \mid x_i \leqslant x_j \mbox{ {\rm for each} } (i,j)\in \rho_\tau \} \, .
\end{equation}
In other words $Cone(\preccurlyeq_\tau)$ is described by all inequalities $x_i \leqslant x_j$, where either $(i,j)\in A_1\times B$ or $(j,i)\in A_2\times B$, and all
equalities $x_i = x_j$ for all pairs $(i,j)\in B\times B$.

\medskip
The original proof (and the statement) of the following theorem is more geometric, emphasising the starshaped embedding $\mathcal{R}_{\pm\delta}(Bier(K))$ of the sphere $Bier(K)$. Here we give a different and more combinatorial proof, which uses the preposet-braid cone dictionary.

\begin{theo}{\rm (\cite[Theorem 3.1]{jevtic_bier_2019})}\label{thm:star}
 Let $K\subsetneq 2^{[n]}$ be a simplicial complex. Then the collection of convex cones
\begin{equation}\label{eqn:fanfan}
     Fan(K) = \{Cone(\preccurlyeq_\tau)\}_{\tau\in Bier(K)}
\end{equation}
 is a complete simplicial fan in $H_0 = \{x\in \mathbb{R}^n \mid x_1+\dots + x_n = 0\}$, referred to as the \emph{canonical fan} associated to $K$. Moreover, the face poset $FaceFan(K)$ is
 isomorphic to the (extended) face poset $FaceBier_0(K)$ of the Bier sphere $Bier(K)$. The construction of the canonical fan is faithful in the sense that if $Fan(K_1) = Fan(K_2)$ then $K_1 = K_2$.
 \end{theo}

\begin{proof}
The faithfulness of the construction is quite immediate, since one can recover both $K$ and $K^\circ$ from the preposets corresponding to maximal
cones in $Fan(K)$. Moreover, the structure of the face poset of $Fan(K)$ is easily recovered from (\ref{eqn:fanfan}).

\medskip
Let us begin the proof that  $Fan(K)$ is a complete, simplicial fan by showing that for each permutation $\pi\in S_n$ there exists exactly one facet $\tau = (A_1, A_2; B) = (A_1, A_2; \{\nu\})$
of the Bier sphere $Bier(K)$ such that,
\begin{equation}
      C_\pi = \{x\in H_0 \mid x_{\pi(1)}\leqslant x_{\pi(2)}\leqslant \dots \leqslant x_{\pi(n)}\} \subseteq Cone(\tau) \, .
\end{equation}
Since $[n]\notin K$ we know that $\{k \mid \{\pi(j)\}_{j\leqslant k} \notin K\} \neq \emptyset$. Let $p = \min\{k \mid
\{\pi(j)\}_{j\leqslant k} \notin K\}$ and let $\nu = \pi(p)$. By construction $A_1:= \{\pi(j)\}_{j< p} \in K$ and
$A_2:= \{\pi(j)\}_{j > p} \in K^\circ$, and it immediately follows that $C_\pi \subseteq Cone(\tau)$ where $\tau = (A_1, A_2; \{\nu\})$.

\medskip Conversely, let us suppose that $Int(C_\pi)\cap Cone(\tau') \neq\emptyset$ where $\tau' = (A_1', A_2', \{\nu'\})\in Bier(K)$.
In other words there exists $x\in Cone(\tau')$ such that
\[
x_{\pi(1)}< x_{\pi(2)}< \dots< x_{\pi(n)} \, .
\]
Let $\nu' = \pi(p')$. Then the condition $x\in Cone(\tau')$ implies that $\{\pi(j)\}_{j< p'} \subseteq A_1' \in K$ and
$\{\pi(j)\}_{j > p'}\subseteq A_2' \in K^\circ$, which immediately implies $p=p'$ and $\tau = \tau'$.

\medskip
If $\tau = (A_1, A_2; B)$ and  $\tau' = (A_1', A_2'; B')$ are two, not necessarily maximal, faces of $Bier(K)$, then
$Cone(\tau) \cap Cone(\tau') = Cone(\tau'')$ where $\tau'' = (A_1'', A_2''; B'')$ is the simplex determined by
the conditions $A_1'' = A_1\cap A_1'$ and $A_2'' = A_2\cap A_2'$. Indeed, this follows from the preposet-braid cone dictionary,
see \cite[Proposition~3.5]{postnikov_faces_2007}, and the following lemma.

\begin{lema}
Transitive closure of the relation $\preccurlyeq_\tau \cup \preccurlyeq_{\tau'}$ coincides with the relation $\preccurlyeq_{\tau''}$. Moreover, $\preccurlyeq_{\tau''}$ is a {\em contraction}  (in the sense of \cite{postnikov_faces_2007}, Section~3.3) of both $\preccurlyeq_{\tau}$ and $\preccurlyeq_{\tau'}$.
\end{lema}

\medskip\noindent
{Proof of Lemma~1:}
Since $\rho_\tau \cup \rho_{\tau'} \subseteq \rho_{\tau''}$ it is sufficient to show that the transitive closure $\preccurlyeq$ of $\preccurlyeq_\tau \cup \preccurlyeq_{\tau'}$ contains the relation $\rho_{\tau''}$. This will follow if we prove that
\begin{equation}\label{eqn:lemma}
 i\preccurlyeq j  \mbox{ {\rm  for each pair of elements in} } B'' = (A_1\Delta A_1')\cup (A_2\Delta A_2')\cup B\cup B' \, .
\end{equation}
As a first step in the proof of (\ref{eqn:lemma}), let us show that $B\cup B'\subseteq B''$.

\medskip
As an immediate consequence of the definition of the Alexander dual $K^\circ$ of a simplicial complex $K$, we obtain the implication
\begin{equation}\label{eqn:transversal}
    X\notin K \mbox{ {\rm and} } Y\notin K^\circ \quad \Rightarrow \quad X\cap Y \neq\emptyset \, .
\end{equation}
From here, in light of $A_1\cup B \notin K$ and $A_2'\cup B' \notin K^\circ$, we deduce $(A_1\cup B)\cap (A_2'\cup B')\neq \emptyset$. Choose $s\in B$ and $t\in B'$ and assume $z\in (A_1\cup B)\cap (A_2'\cup B')$. Then, directly from the definition of preposets $\preccurlyeq_{\tau}$ and $\preccurlyeq_{\tau'}$, we obtain the relation $t\preccurlyeq_{\tau'} z \preccurlyeq_{\tau} s$ and, as a consequence, $t\preccurlyeq s$.

\medskip
Similarly, from $A_1'\cup B' \notin K$ and $A_2\cup B \notin K^\circ$, we deduce that $z'\in (A'_1\cup B)\cap (A_2\cup B)$ for some $z'$.
If $s\in B$ and $t\in B'$ then $s\preccurlyeq_{\tau} z' \preccurlyeq_{\tau'} t$, and as a consequence $s\preccurlyeq t$.
The relations $s\preccurlyeq t$ and $t\preccurlyeq s$ together imply that $s \equiv_\preccurlyeq t$, which completes the proof of the inclusion $B\cup B' \subseteq B''$.

\medskip
For the completion of the proof of (\ref{eqn:lemma}) let us begin with the case $z\in A_1\setminus A_1'$. Then $B' \preccurlyeq_{\tau'} z \preccurlyeq_{\tau} B$ and as a consequence $z\in B''$. Similarly, if $z\in A_2\setminus A_2''$ then $B \preccurlyeq_{\tau'} z \preccurlyeq_{\tau} B'$ and again $z\in B''$. The other two cases $A_1'\setminus A_1\neq\emptyset$ and $A_2'\setminus A_2\neq\emptyset$ are treated analogously.

\medskip For the completion of the proof of Lemma 1 we need to show that both $A_1\cap A_1'$  and $A_2\cap A_2'$ are disjoint from $B''$.
This is obvious since if $z \in A_1\cap A_1'$ ($z \in A_2\cap A_2'$) then $z$ is never a right hand side (respectively left hand side) of a relation involving $\preccurlyeq_{\tau}$ or $\preccurlyeq_{\tau'}$ (except for the trivial relations $z \preccurlyeq_{\tau} z$ and $z \preccurlyeq_{\tau'} z$).
\end{proof}

 \medskip
 The following proposition shows that the fan $Fan(K)$ is isomorphic to the radial fan associated to the starshaped realization
 $\mathcal{R}_{\pm\delta}(Bier(K))$ of the Bier sphere $Bier(K)$, constructed in \cite{jevtic_bier_2019}, Theorem~3.5. The reader is referred to Section \ref{sec:volume} (see also \cite{jevtic_bier_2019}) for all undefined concepts and related facts. In particular the $\delta$-realization is a special case of the $b$-realization from Section \ref{sec:volume} where the vertices $\delta = \{\delta_1,\dots, \delta_{n}\}$ of the selected simplex are the vectors $\delta_i := e_i - (1/n)(e_1+\dots+ e_n)$.

\medskip\noindent
\begin{prop}\label{prop:coincidence}
 The fan $Fan(K)$ coincides with the negative of the fan described in \cite{jevtic_bier_2019}, Theorem~3.5. More explicitly,
\begin{equation}\label{eqn:fans}
    Fan(K)  =  {\rm Cone}_{\mp\delta}(K) = RadialFan(\mathcal{R}_{\mp\delta}(Bier(K)))
\end{equation}
where
\begin{equation}\label{eq:cone}
{\rm Cone}_{\mp\delta}(K) = \{{Cone}(R_{-\delta}(S)\ast R_{\delta}(T)) \mid (S,T)\in K\ast_\Delta K^\circ\} \, .
\end{equation}
\end{prop}

\begin{proof}
Extremal rays of the simplicial cone ${Cone}(R_{-\delta}(S)\ast R_{\delta}(T))$ are generated by the vectors
$\{\hat\delta_i\}_{i\in S}\cup \{\delta_j\}_{j\in T}$, where $\hat\delta_j$ is the barycenter of the facet $\Delta_i\subset \Delta_\delta :={\rm Conv}\{\delta_k\}_{k\in [n]}$,
opposite to the vertex $\delta_i\in \Delta_\delta$.

\medskip Let us show that the extremal rays of the cone $Cone(\tau)$, where $\tau = (S, T; \{\nu\})$, have the same representation. In this case the
preposet $\preccurlyeq_\tau$ (the reflexive closure of  $\rho_\tau = S\times \{\nu\} \cup \{(\nu, \nu)\}\cup \{\nu\}\times T$) is a {\em tree-poset}, in the sense of \cite{postnikov_faces_2007}, Section~3.3, meaning that the associated Hasse
diagram is a spanning three on $[n]$. The corresponding simplicial cone is described by inequalities listed in (\ref{eqn:braid-cone}), and the associated extremal rays
are obtained if all inequalities, with one exception, are turned into equalities.

\medskip If $x_i \leqslant x_\nu$ is the excepted inequality (where $i\in S$), then the corresponding ray has a parametric representation
$x_k = t$ for $k\neq i$ and $x_i = -(n-1)t \leqslant x_\nu = t$. From here it immediately follows that this ray is spanned by $\hat\delta_i$.
If $x_\nu \leqslant x_j$ is the excepted inequality (where $j\in T$), then the corresponding ray has a parametric representation
$x_k = t$ for $k\neq j$ and $x_j = -(n-1)t \geqslant x_\nu = t$. In this case the spanning vector is $\delta_j$.
\end{proof}

\section{Volume and $f$-vector of Bier spheres}
\label{sec:volume}

Bier sphere $Bier(K)$, being an abstract simplicial complex,  must be realized as a geometric sphere in order to discuss the volume of its inner region. The geometric realization $Star(K)$, considered here, is the convex body in $H_0$ (with the apex at the origin) whose boundary $\partial Star(K) = \mathcal{R}_{\pm \delta}(Bier(K))$ is the starshaped embedding of the Bier sphere originally described in \cite{jevtic_bier_2019} (see equation  (\ref{eqn:b-rep}) for an explicit definition).

Let $K\subset 2^{[n]}$ be a simplicial complex and $K^\circ$ its Alexander dual. The ``naive'' or \emph{tautological geometric realization} of the Bier sphere is the embedding in $\mathbb{R}^n$
\begin{equation}
Bier(K):= K \ast_\Delta K^\circ \hookrightarrow \Delta_{n} \ast_\Delta \Delta_{\bar{n}} \cong \partial\lozenge_{n} \subset \mathbb{R}^n
\end{equation}
arising from the standard geometric realizations $K\hookrightarrow \Delta_{n}\, (K^\circ \hookrightarrow \Delta_{\bar{n}} := -\Delta_{n})$, where $\Delta_{n} := {\rm Conv}\{e_i\}_{i=1}^n$ and $\partial\lozenge_{n}$ is the boundary sphere of the cross-polytope $\lozenge_{n}:= {\rm Conv}\{\pm e_i\}_{i=1}^n$.

\medskip

Let $b_1,\ldots,b_{n} \in H_0\cong \mathbb{R}^{n-1}$, $\sum_{i=1}^{n} b_i=0$, be the vertices of a $(n-1)$-dimensional simplex which has the barycenter at the origin. (A canonical choice is the simplex spanned by vertices $\delta_1,\dots, \delta_n$, used in Proposition \ref{prop:coincidence}.)
For $S \subseteq [n]$, the corresponding $b$-representation is the geometric simplex
\begin{align*}
R_{b}(S)=Conv\{ b_i\}_{i\in S} \, .
\end{align*}
Following \cite{jevtic_bier_2019} the $b$-representation of $Bier(K)$ is the starshaped sphere
\begin{equation}\label{eqn:b-rep}
\mathcal{R}_{\pm b}(Bier(K)) = \bigcup\{R_b(S)\ast R_{-b}(T) \mid (S, T)\in K\ast_\Delta K^\circ\} \subset H_0\, .
\end{equation}
We are interested in the volume of the associated starshaped body $Star(K) := \{0\} \ast \mathcal{R}_{\pm b}(Bier(K))$ (the geometric join of the starshaped sphere with the origin).

Let $\tau = (S,T;\{i\})$ be a facet of $Bier(K)$. Then $R(\tau) = R_b(S)\ast R_{-b}(T)$, the corresponding geometric simplex from (\ref{eqn:b-rep}), contributes to the volume of $Star(K)$ the quantity  $Vol_\tau$ where
\begin{align*}
  (n-1)!\, Vol_\tau := |{\rm Det}(\tau)| =  |\epsilon_1b_1 \ldots \hat{b}_i \ldots \epsilon_nb_n|
\end{align*}
($\epsilon_i = +1$ if $i\in S$ and $\epsilon_i = -1$ if $i\in T$) and the volume of $Star(K)$ is
\begin{align}\label{align:vol_tau}
    Vol(Star(K))=\sum_{\tau} Vol_\tau \, .
\end{align}
Notice that $Vol_0 = Vol_\tau$ is a constant, independent of the facet $\tau\in Bier(K)$.   Let
\begin{align*}
    m_i(K) =m_i = |\{S \in K \mid S \cup \{i\} \not \in K\}| \, .
\end{align*}
In light of (\ref{align:vol_tau}) the volume of $Star(K)$ can be calculated as
\begin{align}\label{align:vol_0}
    Vol(Star(K))=Vol_0 \sum_{i=1}^{n}m_i = Vol_0f_{n-1}(Bier(K))
\end{align}
where $f_{n-1}(Bier(K))$ is the number of facets of the Bier sphere $Bier(K)$.

\medskip
The following proposition allows us to compare the volumes of Bier spheres which are obtained one from the other by a \emph{bistellar operation}, see  \cite[Section 5.6]{matousek_using_2008}.

\begin{prop} \label{prop:vol-3-cases}
Assume that $K \subsetneq 2^{[n]}$ is a simplicial complex and let $Star(K)\subset H_0$ be the associated starshaped body.
Let $B \not \in K$ be a minimal non-face of $K$ in the sense that $\left( \forall i \in B \right) B \setminus \{i\} \in K$, and let $K'=K\cup\{B\}$. Let $C = [n]\setminus B$ the complement of $B$.
Then
\begin{align*}
    Vol\left( Star(K') \right) - Vol\left( Start(K) \right)= V(K',K) = (\vert B\vert - \vert C\vert)Vol_0 \, .
\end{align*}
The following relations are an immediate consequence
\begin{align*}
    V(K',K)&>0, \,\, \textrm{if } |B|<\frac{n}{2}\\
    V(K',K)&=0, \,\, \textrm{if } |B|=\frac{n}{2}\\
    V(K',K)&<0, \,\, \textrm{if } |B|>\frac{n}{2}
\end{align*}
\end{prop}

\begin{proof}
  Let $\Sigma = \Sigma_b = R_b(B)\ast R_{-b}(C)$ be the (possibly degenerate) simplex in $H_0$ which has $R_b(B)$ and $R_{-b}(C)$ as two ``complementary faces''.
(Note that $\Sigma$ is degenerate precisely if $\vert B\vert = \vert C\vert = n/2$ in which case the simplices $R_b(B)$ and $R_{-b}(C)$ intersect in a common barycenter.)

 If $\Sigma$ is non-degenerate its boundary $\Sigma$ is
  the union of two discs
  \[
       \partial\Sigma = \partial(R_b(B)\ast R_{-b}(C)) = (\partial(R_b(B))\ast R_{-b}(C)) \cup (R_b(B)\ast \partial(R_{-b}(C))) = \Sigma_1\cup \Sigma_2
  \]
where $\Sigma_1\subseteq Bier(K')$ and $\Sigma_2\subseteq Bier(K)$. If $\Sigma$ is degenerate then $\Sigma = \Sigma_1 = \Sigma_2$ (as sets), more precisely $\Sigma_1$ and $\Sigma_2$ are two different triangulations of $\Sigma$.

Note that  $Bier(K')\setminus \Sigma_1 = Bier(K)\setminus \Sigma_2$ and $Cone(\Sigma_1) = Cone(\Sigma_2)= Cone(\Sigma)$. From here we observe that

\begin{enumerate}
  \item $Star(K) = Star(K')$ if and only if $\vert B\vert = \vert C\vert = n/2$;
  \item $Star(K) \subsetneq Star(K')$ if and only if $\vert B\vert > \vert C\vert$;
  \item $\vert V(K, K')\vert = \vert\vert B\vert - \vert C\vert\vert Vol_0 = Vol(\Sigma)$.
\end{enumerate}
For example the third relation is a consequence of (\ref{align:vol_0}) or can be deduced directly by a similar argument.  \end{proof}

\begin{prop}\label{prop:max-vol}
 If $n=2m+1$ is odd the unique Bier sphere of maximal volume is $Bier(K)$ where
 \begin{equation}\label{eqn:VK-F-1}
K =\binom{[n]}{\leqslant m} = \{ S\subset [n] \mid \vert S\vert \leq m \} \, .
\end{equation}
 If $n=2m$ is even a Bier sphere $Bier(K)$ is of maximal volume if and only if
 \begin{equation}\label{eqn:VK-F-2}
 \binom{[n]}{\leqslant m-1} \subseteq K \subseteq \binom{[n]}{\leqslant m} \, .
 \end{equation}
 A Bier sphere $Bier(K)$ is of minimal volume if and only of $K = \partial\Delta_{[n]}= 2^{[n]}\setminus \{[n]\}$ is either the boundary of the simplex $\Delta_{[n]}$ or $K = \{\emptyset\}$.
\end{prop}

\begin{proof}
The first half of proposition, describing the Bier spheres of maximal volume, is an immediate consequence of Proposition \ref{prop:vol-3-cases}. The second, describing the Bier spheres of minimal volume, is an immediate consequence of the formula (\ref{align:vol_0}), since the unique triangulation of a sphere $S^{m-1}$ with the minimum number of facets is the boundary of an $m$-dimensional simplex.
\end{proof}


\begin{cor}\label{cor:VK-F}
For all Bier spheres $Bier(K)$ of maximal volume, the convex body $\Omega_n = Star(K)$ is unique and independent of $K$. The body $\Omega_n$ is centrally symmetric. More explicitly $\Omega_n = {\rm Conv}(\Delta_\delta \cup \nabla_\delta)$ where $\Delta_\delta \subset H_0$ is the simplex spanned by vertices $\delta_i := e_i - {\frac{1}{n}}(e_1+\dots+ e_n)$ and $\nabla_\delta := -\Delta_\delta = \Delta_{\bar\delta}$ is the simplex spanned by $\bar\delta_i = -\delta_i$.
 The centrally symmetric convex body $\Omega_n$ is from here on referred to as the \emph{Van Kampen-Flores polytope} in dimension $n$.
\end{cor}

\begin{proof}
The body $\Omega_n$ is centrally symmetric since the sphere centrally symmetric to the Bier sphere $Bier(K)$ is the sphere $Bier(K^\circ)$ and $\Omega_n = Star(K) = Star(K^\circ)$ if $K$ is one of the complexes described in equations (VK-F-1) and (VK-F-2). More  precisely $\Omega_n = {\rm Conv}(\Delta_\delta \cup \nabla_\delta)$ since
\[
  \bigcup Star(K) = {\rm Conv}(\Delta_\delta \cup \nabla_\delta)
\]
where the union on the left is taken over all simplicial complexes $K\subsetneq 2^{[n]}$.
\end{proof}
We call $\Omega_n$ the \emph{Van Kampen-Flores body} (polytope) in dimension $n$ for the following reason.
The Bier sphere of the simplicial complex (\ref{eqn:VK-F-1}) is precisely the simplicial triangulation of the $(m-2)$-sphere, used it the standard proof of the classical Van Kampen-Flores theorem, which claims that the $(m-1)$-dimensional complex $\binom{[2m+1]}{\leqslant m}$ is not embeddable in $\mathbb{R}^{2m-2}$ (see \cite[Section 5.6]{matousek_using_2008}).

The complexes $\binom{[2m]}{\leqslant m-1}$  and $\binom{[2m]}{\leqslant m}$ (the boundary complexes mentioned in (\ref{eqn:VK-F-2})) appear in the \emph{``sharpened Van Kampen-Flores theorem''}  (Theorem~6.8 from \cite{blagojevic_tverberg_2014}).

Finally all complexes mentioned in (\ref{eqn:VK-F-1}) and (\ref{eqn:VK-F-2}) appeared under the name  \emph{balanced complexes} in the following theorem, which unifies and extends previously known results.

\begin{theo}\label{thm:seems} {\rm (\cite[Theorem 3.5]{jojic_tverberg_2021})}
Let $K\subset 2^{[n]}$ be a simplicial complex and let $K^\circ$ be its Alexander dual. Assume that $K$ is balanced in the sense that either (\ref{eqn:VK-F-1}) or (\ref{eqn:VK-F-2}) is satisfied. Then for each continuous map $f : \Delta^{n-1} \rightarrow \mathbb{R}^{n-3}$ there
exist disjoint faces $F_1\in K$ and $F_2\in K^\circ$ such that $f(F_1) \cap f(F_2) \neq \emptyset$.
\end{theo}
The importance of complexes listed in  equations  (\ref{eqn:VK-F-1}) and (\ref{eqn:VK-F-2}) in    Proposition \ref{prop:max-vol} was noted even earlier. In \cite[Section 5.6]{matousek_using_2008} they were used as a source of examples of non-polytopal triangulations of spheres while in  \cite{bjorner_bier_2004} they provided examples of \emph{nearly neighborly Bier spheres}.

\section{Van Kampen-Flores polyhedra and median hypersimplices}
\label{sec:hypersimplex}

 The Van Kampen-Flores polytope was introduced in the previous section as the convex hull
 \begin{align*}
    \Omega_n = {\rm Conv}(\Delta\cup\nabla) = {\rm Conv}\left\lbrace u_1,u_2,\dots, u_n, -u_1, -u_2, \dots, -u_n\right\rbrace
\end{align*}
 where $\Delta = {\rm Conv}\{u_i\}_{i=1}^n \subset \mathbb{R}^{n-1}$ is a non-degenerate simplex with barycenter at the origin and $\nabla := -\Delta$ the opposite simplex.

Recall that  a \emph{circuit} in $\mathbb{R}^{n-1}$ is a spanning family $\{u_i\}_{i=1}^n$ in $\mathbb{R}^{n-1}$ such that $u_1+\dots+ u_n = 0$. In other words  $\{u_i\}_{i=1}^n\subset \mathbb{R}^{n-1}$ is a \emph{circuit} if
the linear map  \[\mathbb{R}^n \stackrel {\Lambda}{\longrightarrow} \mathbb{R}^{n-1}, \,  \lambda = (\lambda_1, \dots, \lambda_n) \mapsto \Lambda(\lambda) := \lambda_1u_1+\dots+\lambda_nu_n\] is an epimorphism with the kernel generated by $\mathbbm{1} = (1,1,\dots, 1)\in \mathbb{R}^n$.

 \medskip
 The polytope $\Omega_n$ must have been well-known, in this or equivalent form, in classical theory of convex polytopes, although, perhaps, without a specific name. In \cite[Theorem 2.2]{jevtic_bier_2019} it originally appeared as a member of the family $Q_{L,\alpha} = {\rm Conv}(\Delta_L \cup -\alpha\Delta_L)$
of polytopes where $\Delta_L = {\rm Conv}\{l_1u_1,\dots, l_nu_n\}$ is a radial perturbation of $\Delta$ (for some positive weight vector $L = (l_1,\dots,l_n)$) and $\alpha>0$.

\medskip
The results from Section \ref{sec:volume} provide, in our opinion, a sufficient evidence that this polytope may deserve  some independent interest. For this reason, and for future reference, we collect here some basic information  about the facial structure of the Van Kampen-Flores polytope and its polar dual.

 \begin{prop}
 The set $Vert(\Omega_n) = \{u_1,u_2,\dots, u_n, -u_1, -u_2, \dots, -u_n\}$ is clearly the vertex set of the polytope $\Omega_n$. More generally, a subset $\{u_i\}_{i \in I} \cup \{-u_j\}_{j\in J} \subset Vert(\Omega_n)$ corresponds to a proper face of $Q$ if and only if
\begin{align*}
    I\cap J = \emptyset \quad \mbox{and} \quad |I|,|J|\leqslant\frac{n}{2}.
\end{align*}
 \end{prop}

\begin{proof}
Let  $z : \mathbb{R}^{n-1} \rightarrow  \mathbb{R}$ be a non-zero linear form such that the associated hyperplane
$H_z := \{x\in \mathbb{R}^{n-1} \mid \langle z, x \rangle = 1\}$ is a supporting hyperplane of $\Omega_n$. The
corresponding face of the polytope $\Omega_n$ is described by a pair $(I, J)$ of subsets of $[n]$ recording which
vertices of the polytope $\Omega_n$ belong to the hyperplane $H_z$. More explicitly
\begin{align*}
           \Omega_n \cap H_z = {\rm Conv}(\{u_i\}_{i\in I}\cup\{ u_j\}_{j\in J}) \, .
\end{align*}
The ordered pair $(I,J)$ of subsets of $[n]$ must satisfy the following

\begin{align}
(\forall i\in I) \, \langle z, u_i \rangle = 1   &&    (\forall j\in J) \, \langle z, -u_j \rangle = 1 \label{eqn:prva} \\
(\forall k\notin I) \, \langle z, u_k \rangle < 1   &&    (\forall k\notin J) \, \langle z, -u_k \rangle < 1 \label{eqn:druga}
\end{align}

Clearly $I$ and $J$ have to be disjoint. Let $a_i=\langle z, u_i \rangle$. From the previous equations it follows that $a_i \in [-1,1]$. Therefore, if $|I|>\frac{n}{2}$ it would follow that
\begin{align*}
0=\sum_{i=1}^n a_i = |I| + \sum_{i \in I^c} u_i > |I|-|I^c|>0
\end{align*}
which is a contradiction. Hence, $|I| \leqslant \frac{n}{2}$.
Conversely, if $|I|,|J|\leqslant \frac{n}{2}$ the existence of $z$ which satisfies conditions (\ref{eqn:prva}) and (\ref{eqn:druga}) is guaranteed by
the dimension argument.
\end{proof}

\medskip
We turn our attention now to the polar polytope $R_n:=\Omega_n^\circ$ of the Van Kampen-Flores polytope. As visible from Figure \ref{fig:prva}, in the case $n=4$ the polytope $\Omega_4$ is the three dimensional cube while $Q_4^\circ$ is the octahedron.

\bigskip
Recall that the {\em Minkowski functional} $\mu_K$, associated to a convex body $K\subseteq \mathbb{R}^{n-1}$ (which contains the origin in its interior),
is the convex function $\mu_K : \mathbb{R}^{n-1}\rightarrow \mathbb{R}$, defined by
\[
\mu_K(x) = d(0, x)/d(0,x_0) = \mbox{\rm Inf}\{r>0 \mid x\in rK \}
\]
where $d(\cdot, \cdot)$ is the Euclidean distance function and $x_0$ is the intersection of the positive ray through $x$ and the boundary of $K$.

The following proposition determines the polar dual of a convex body $K$, from the Minkowski functional $\mu_K$, as the set $K^\circ = \{x \mid \mu_K(x)\leq 1\}$.
\begin{prop}
Minkowski functional of a convex body $K$ is equal to the support functional of its polar dual
\[
          \mu_K = h_{K^\circ} \, .
\]
\end{prop}
The following relation (for two convex bodies $K$ and $L$ containing the origin in their interior) follows directly from the definition
\begin{equation}\label{eqn:mink-presek}
  \mu_{K\cap L} = \max\{\mu_K, \mu_L\} \, .
\end{equation}
Let us calculate the Minkowski functional of the polytope $R_n = \Omega_n^\circ$. Since
\[
   ({\rm Conv}(K\cup L))^\circ = K^\circ \cap L^\circ
\]
and $\Delta^\circ = \nabla, \nabla^\circ = \Delta$ we observe that
\[
\Omega_n^\circ = ({\rm Conv}(\Delta\cup\nabla))^\circ \cong \nabla\cap \Delta \, .
\]

We use basic properties of the  functions $x^+ = \max \{0, x\}$ and $x^- = \max \{0, -x\} = (-x )^+$, which satisfy the well-known elementary relations
\[
\begin{array}{ccc}
x = x^+ - x^- &  & \vert x\vert = x^+ + x^-  \\
x^+ = \frac{1}{2}(\vert x\vert + x) &  & x^- = \frac{1}{2}(\vert x\vert - x) \, . \end{array}
\]

Each vector $x\in \mathbb{R}^{n-1}$ has a unique representation
\[
   x = \lambda_1u_1+\lambda_2u_2+\dots+ \lambda_nu_n
\]
 where $\lambda_1+\dots+\lambda_n = 0$ .

 \begin{prop}\label{prop:Mink-lepo}
   The Minkowski functionals of simplices $\Delta$ and $\nabla$, and of their intersection $\Omega_n^\circ = \Delta\cap \nabla$ are the following
   \[
      \mu_\Delta(x) = n\max\{\lambda_i^-\}_{i=1}^n \qquad \mu_\nabla(x) = n\max\{\lambda_i^+\}_{i=1}^n \qquad \mu_{\Omega_n^\circ}(x) = n\max\{\vert\lambda_i\vert\}_{i=1}^n \, .
   \]
  \end{prop}

 \begin{proof}
 Assuming that $x = \lambda_1a_1+\dots+\lambda_na_n\neq 0$, let us calculate the corresponding point $x_0\in \partial(\Delta)\cap {\rm Ray}(0,x)$, defined as the intersection point of the boundary of $\Delta$ with the ray emanating from the origin $0$, passing through the point $x$.

  If $\lambda := \max\{\lambda_i^-\}_{i=1}^n$ then
  \[
            x = (\lambda + \lambda_1)a_1 + \dots + (\lambda +\lambda_n)a_n
  \]
where $\lambda + \lambda_i \geq 0$ for each $i\in [n]$ and $\lambda + \lambda_j = 0$ for at least one $j\in [n]$.  A moment's reflection shows
\[
            x_0 = \frac{x}{n\lambda} \in \partial(\Delta)
\]
 which immediately implies that $\mu_\Delta(x) = n\lambda = n\max\{\lambda_i^-\}_{i=1}^n$.

Since $\mu_{-K}(x) = \mu_K(-x)$ we observe that
\[
  \mu_\nabla(x) = \mu_\Delta(-x) = n\max\{(-\lambda_i)^-\}_{i=1}^n = n\max\{(\lambda_i)^+\}_{i=1}^n \, .
\]
The third formula $\mu_{\Omega_n^\circ}(x) =  n\max\{\vert\lambda_i\vert\}_{i=1}^n$ is an immediate consequence of   (\ref{eqn:mink-presek}) and the relation $\max\{\lambda_i^+, \lambda_i^-\} = \vert \lambda_i\vert$.
\end{proof}

\medskip
Since $K = \{x\in \mathbb{R}^{n-1} \mid \mu_K(x) \leq 1\}$, as a corollary of Proposition \ref{prop:Mink-lepo} we obtain the following result.

\begin{cor}\label{cor:R_n}
\[
 \Omega_n^\circ = {\rm Conv}(\Delta\cap \nabla) =  \{ x = \lambda_1a_1+\dots \lambda_na_n \vert \, \lambda_1+\dots+ \lambda_n = 0 \mbox{ {\rm and} } (\forall i)\, \vert \lambda_i\vert \leq 1 \} \, .
\]
\end{cor}

\begin{defin}\label{def:hypersimplex}
  A hypersimplex $\Delta_{n,r}$ with parameters $n,r$ is  defined as the convex hull of all $n$-dimensional vectors, vertices of the $n$-dimensional cube $[0,1]^n$, which belong to the hyperplane $x_1+\dots+ x_n = r$. Alternatively $\Delta_{n,r} = {\rm Newton}(\sigma_r)$ can be described as the Newton polytope of the elementary symmetric function $\sigma_r$ of degree $r$ in $n$ variables.
\end{defin}

\begin{theo}\label{thm:R_n=hypersimplex}
  If $n=2k$ is even then $\Omega_{2k}^\circ = \Delta\cap \nabla$ is affine isomorphic to the hypersimplex $\Delta_{2k,k}$.  If $n=2k+1$ then $\Omega_n^\circ$ is affine isomorphic to the convex hull
  \begin{equation}\label{eqn:hull}
    \Omega_{2k+1}^\circ \cong {\rm Conv}\{\lambda \in [0,1]^{2k+1}\mid \, (\forall i)\, \lambda_i\in\{0,1/2,1\}, \, \vert\{j \mid \lambda_j = 0\}\vert = \vert\{j \mid \lambda_j = 1\}\vert = k  \} \, .
  \end{equation}
\end{theo}

\begin{proof}
By Corollary \ref{cor:R_n} for each circuit $\{a_i\}_{i=1}^n$ the polytope   $R_n$ is affine isomorphic to the intersection of the hyperplane $\lambda_1+\dots+ \lambda_n = 0$ with the $n$-cube $[-1,+1]^n$. The (inverse of the) affine transformation $\lambda_i = 2x_i  -1 \, (i=1,\dots, n)$  maps this to the intersection of the hypercube $[0,1]^n$ with the hyperplane $x_1+\dots+x_n = n/2$.

If $n=2k$ we obtain the hypersimplex $\Delta_{2k,k} $. If $n=2k+1$ we obtain the polytope (\ref{eqn:hull}).
\end{proof}

\section{Wall crossing functions}
\label{sec:wall-crossing}

In this section we return to the question of polytopality of Bier spheres. The main result is a \emph{$K$-submodularity theorem} which for polytopal Bier spheres plays the role similar to the role of classical submodular functions (polymatroids) in the theory of generalized permutohedra.

\begin{prop}\label{prop:wall-crossing}{\rm (\cite{albertin_removahedral_2020})}
Let $\mathcal{F}$ be an essential complete simplicial fan in
$\mathbb{R}^n$ and $\mathbf{G}$ be the $N\times n$ matrix whose rows are
the rays of $\mathcal{F}$. Then the following are equivalent for any
vector $\mathbf{h} \in \mathbb{R}^N$. \begin{enumerate}[label=(\arabic*)]
     \item[{\rm (1)}] The fan $\mathcal{F}$ is the normal fan of the polytope
$P_{\mathbf{h}}:=\{x \in \mathbb{R}^n \mid \mathbf{G}x \leqslant
\mathbf{h} \}$.
     \item[{\rm (2)}] For any two adjacent chambers $\mathbb{R}_{\geqslant
0}\mathbf{R}$ and $\mathbb{R}_{\geqslant 0}\mathbf{S}$ of $\mathcal{F}$
with $\mathbf{R}\setminus \{r\}=\mathbf{S}\setminus \{s\}$,
     \begin{align}\label{eqn:wall-inequality}
         \alpha\mathbf{h_r}+\beta\mathbf{h_s}+\sum_{\mathbf{t}\in
\mathbf{R}\cap \mathbf{S}} \gamma_{\mathbf{t}}\mathbf{h_t}>0,
     \end{align} where \begin{align}\label{eqn:wall-equality}
         \alpha\mathbf{r}+\beta\mathbf{s}+\sum_{\mathbf{t}\in
\mathbf{R}\cap \mathbf{S}} \gamma_{\mathbf{t}}\mathbf{t}=0
     \end{align}
     is the unique (up to scaling) linear dependence with $\alpha,\beta>0$
between the rays of $\mathbf{R} \cup \mathbf{S}$. \end{enumerate}
\end{prop}

\begin{defin}
Given a (proper) simplicial complex $K\subsetneq 2^{[n]}$,  an element $A\in K$ is a \emph{boundary simplex} if\ $(\exists c\in [n])\, A\cup\{c\}\notin K$. Similarly $B\notin K$ is a \emph{boundary non-simplex} if \ $(\exists c\in [n])\, B\setminus \{c\}\in K$. A pair $(A, B')\in (K,2^{[n]}\setminus K)$ is a \emph{boundary pair} if $B' = A\cup\{c\}$ for some $c\in [n]$.
\end{defin}
We already know (Section \ref{sec:Bier_fans}) that boundary pairs $(A,B')$ correspond to maximal simplices in $Bier(K)$. In the following proposition we describe the ridges, i.e.\ the codimension one simplices in the Bier sphere $Bier(K)$.

\begin{prop}\label{prop:ridges}
The ridges (codimension one simplices) $\tau \in Bier(K)$  have one of the following three forms, exhibited in Figure \ref{ex-3}. Here we use the \emph{interval notation} $\tau = (X,Y)$ (Section \ref{sec:Bier_fans}) where $X\subsetneq Y, X\in K, Y\notin K$ and $(X,Y) \neq (\emptyset, [n])$.
\end{prop}

\begin{figure}[htb]
    \centering
    \subfigure[$\Lambda$ configuration]{\begin{tikzpicture}[scale=0.9]
    \node (y) at (0,2) {$Y=X\cup \{c_1,c_2\}$};
    \node (x1) at (-2,0) {$X_1=X\cup \{c_1\}$};
    \node (x2) at (2,0) {$X_2=X\cup \{c_2\}$};
    \node (x) at (0,-2) {$X$};

    \draw (x) -- (x1) -- (y);
    \draw (x) -- (x2) -- (y);

    \draw[dashed] (-1.5,1) -- (1.5,1);
    \node[anchor=west] (bier) at (1.5,1) {$Bier(K)$};
\end{tikzpicture}}
    \subfigure[$V$ configuration]{\begin{tikzpicture}[scale=0.9]
    \node (y) at (0,2) {$Y=X\cup \{c_1,c_2\}$};
    \node (x1) at (-2,0) {$X_1=X\cup \{c_1\}$};
    \node (x2) at (2,0) {$X_2=X\cup \{c_2\}$};
    \node (x) at (0,-2) {$X$};

    \draw (x) -- (x1) -- (y);
    \draw (x) -- (x2) -- (y);

    \draw[dashed] (-1.5,-1) -- (1.5,-1);
    \node[anchor=west] (bier) at (1.5,-1) {$Bier(K)$};
\end{tikzpicture}}
    \subfigure[$X$ configuration]{\begin{tikzpicture}[scale=0.9]
    \node (y) at (0,2) {$Y=X\cup \{c_1,c_2\}$};
    \node (x1) at (-2,0) {$X_1=X\cup \{c_1\}$};
    \node (x2) at (2,0) {$X_2=X\cup \{c_2\}$};
    \node (x) at (0,-2) {$X$};

    \draw (x) -- (x1) -- (y);
    \draw (x) -- (x2) -- (y);

    \draw[dashed] (1.5,1.5) -- (-1.5,-1.5);
    \node[anchor=west] (bier) at (1.5,1.5) {$Bier(K)$};
\end{tikzpicture}}
    \caption{Configurations of maximal adjacent simplices in $Bier(K)$.}
    \label{ex-3}
\end{figure}
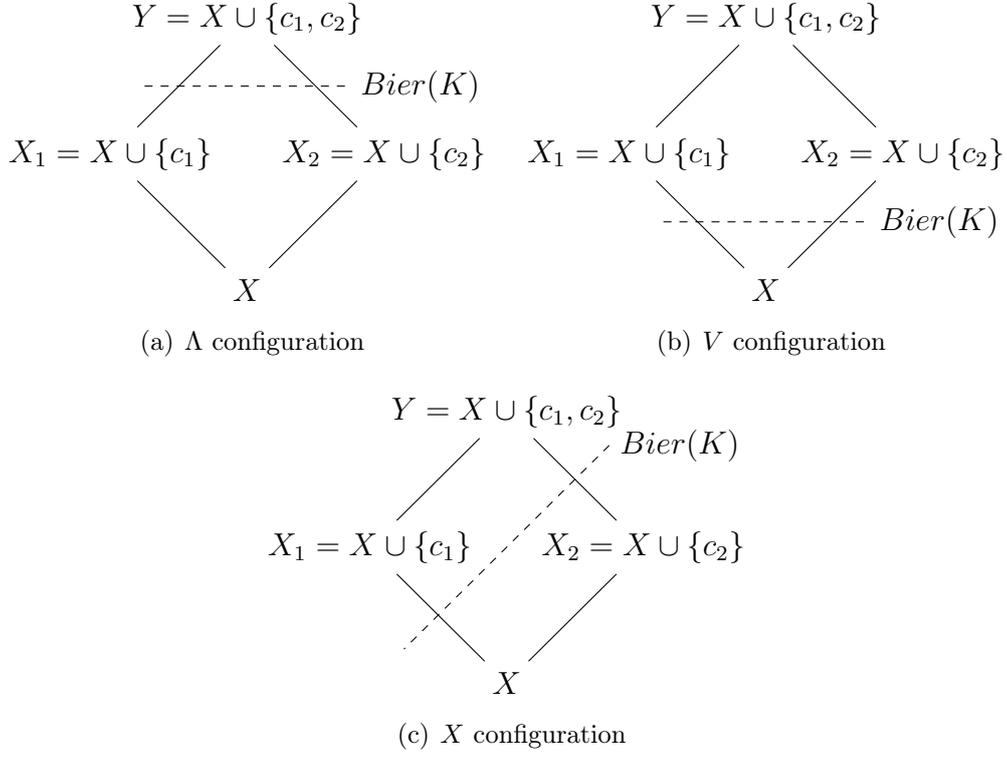

\begin{proof} In the interval notation, the ridges in $Bier(K)$ correspond to intervals $(X,Y)$ where $Y= X\cup \{c_1, c_2\}$ and $c_1\neq c_2$. The $\Lambda$-configurations correspond to the case when both $X_1$ and $X_2$ are in $K$, the $V$-configurations  correspond to the case when neither $X_1$ nor $X_2$ are in $K$, and the $X$-configurations arise if precisely one of these sets is in $K$.
\end{proof}

\begin{defin}\label{def:K-submodular}
Let $K\subsetneq 2^{[n]}$ be a simplicial complex and $Bier(K)$ the associated Bier sphere. A \emph{$K$-submodular function} ($K$-wall crossing function) is a function $f : Vert(Bier(K))\rightarrow \mathbb{R}$ such that
\begin{align}
f(c_1) + f(c_2) + \Sigma_{i\in X} f(i) > \Sigma_{j\notin Y} f(\bar{j}) & \mbox{\quad {\rm for each $\Lambda$-configuration} } \label{eq:K-1}\\ f(\bar{c}_1) + f(\bar{c}_2) + \Sigma_{j\notin X} f(\bar{j}) > \Sigma_{i\in X} f(i) & \mbox{\quad {\rm for each $V$-configuration} } \label{eq:K-2}\\
f(c_2) + f(\bar{c}_2) > 0 & \mbox{\quad {\rm for each $X$-configuration}.} \label{eq:K-3}
\end{align}
\end{defin}

\begin{theo}\label{thm:K-submodular}
  Let $\mathcal{F}= Fan(K)$ be the radial fan arising from the canonical starshaped realization of the associated Bier sphere $Bier(K)$. (The fan $\mathcal{F}$ is by Theorem \ref{thm:star} a coarsening of the braid fan.) Then $\mathcal{F}$ is a normal fan of a convex polytope if and only if the simplicial complex $K$ admits a $K$-submodular function. Moreover, there is a bijection between convex realizations of $Bier(K)$ with radial fan $\mathcal{F}$ and $K$-submodular functions $f$.

  \begin{proof}
  We apply Proposition \ref{prop:wall-crossing} to the fan $\mathcal{F}= Fan(K)$.

  Let $\delta = (\delta_1,\dots, \delta_n)$ be a circuit in $H_0$ where $\delta_i = e_i-\frac{u}{n}\, (u = e_1+\dots+ e_n)$. Let
  $\bar\delta = (\bar\delta_1,\dots, \bar\delta_n)$ be the opposite circuit where $\bar{\delta}_i := -\delta_i$. The vertices of $Bier(K)$ are $\{1,\dots, n, \bar{1}, \dots, \bar{n}\}$ and for the corresponding representatives on the one dimensional cones of the fan $\mathcal{F}= Fan(K)$ we choose $\{\delta_1,\dots, \delta_n, \bar\delta_1,\dots, \bar\delta_n\}$.

  Our objective is to identify the corresponding ``wall crossing relations'' (\ref{eqn:wall-equality}), in each of the three cases listed in Figure \ref{ex-3}, and to read off the associated ``wall crossing inequalities'' (\ref{eqn:wall-inequality}).

  \medskip
  In order to identify the wall crossing relations in the case of the $\Lambda$ and $V$ configurations we observe that, if $[n] = S\cup T$ and $S\cap T=\emptyset$ then, up to a linear factor, the only dependence in the set $\{\delta_i\}_{i\in S}\cup \{\bar\delta_j\}_{j\in T}$ is the relation
 \[
 \sum_{i\in S} \delta_i = \sum_{j\in T} \bar\delta_j \, .
 \]
The first two inequalities in Definition \ref{def:K-submodular} are an immediate consequence. To complete the proof it is sufficient to observe that, in the case of an $X$ configuration, the only dependence in the set $\{\delta_i\}_{i\in X}\cup \{\bar\delta_j\}_{j\notin Y}\cup\{\delta_{c_2}, \bar\delta_{c_2}\}$ is, up to a non-zero factor, the relation $\delta_{c_2} + \bar\delta_{c_2} = 0$.
  \end{proof}
\end{theo}
As an illustration we use Theorem \ref{thm:K-submodular} to show that Bier spheres of threshold complexes are polytopal. This result was originally obtained in \cite{jevtic_bier_2019} (Theorem 2.2) by a different method.

\medskip
Suppose that $L = (l_1,l_2,\dots, l_n)\in \mathbb{R}^n_+$
is a strictly positive vector. The associated measure (weight distribution) $\mu_L$ on $[n]$ is defined by
$\mu_L(I) = \sum_{i\in I}~l_i$ (for each $I\subseteq [n]$).

\medskip
Given a threshold $\nu>0$, the associated threshold complex is $T_{\mu_L < \nu} := \{I\subseteq [n] \mid \mu_L(I)< \nu\}$.
Without loss of generality we assume that $\mu_L([n]) = l_1+\dots+ l_n = 1$. Moreover (\cite{jevtic_bier_2019}. Remark 2.1) we can always assume, without loss of generality, that $\mu_L(I)\neq \nu$ for each $I\subseteq [n]$, which implies that the Alexander dual of $K$ is $K^\circ = T_{\mu_L \leqslant 1- \nu} = T_{\mu_L < 1- \nu}$.

\begin{cor}{\rm (\cite{jevtic_bier_2019}, Theorem 2.2)} $Bier(T_{\mu_L<\nu})$ is isomorphic to the boundary sphere of a convex polytope which can be realized as a polar dual of a generalized permutohedron.
\end{cor}
\begin{proof}
Following Theorem \ref{thm:K-submodular}, it is sufficient to construct a $K$-submodular function $f : [n]\cup [\bar{n}]\rightarrow \mathbb{R}$ where $[n]\cup [\bar{n}] = Vert(Bier(K)) = \{1,\dots, n, \bar{1}, \dots, \bar{n}\}$.
Let us show that the function defined by
\begin{equation}\label{eq:f-threshold}
  f(i) = (1-\nu)l_i \qquad f(\bar{j}) = \nu l_j \qquad (i,j = 1,\dots, n)
\end{equation}
is indeed $K$-submodular for $K = T_{\mu_L<\nu}$. The inequalities \eqref{eq:K-1} and \eqref{eq:K-2}, for the function $f$ defined by \eqref{eq:f-threshold}, take (in the notation of Definition \ref{def:K-submodular} and Figure \ref{ex-3}) the following form
\begin{equation}\label{eq:both}
   \nu\mu_L(Y) > (1-\nu)\mu_L(Y^c)  \qquad (1-\nu)\mu_L(X) < \nu\mu_L(X^c) \, .
\end{equation}
However, in a threshold complex, both inequalities  \eqref{eq:both} hold without any restrictions on a simplex $X\in K$ and a non-simplex $Y\notin K$. (For example the second inequality in \eqref{eq:both} is a consequence of $\mu_L(X)< \nu$ and $\mu(X^c) > 1-\nu$.)

The convex polytope obtained by this construction is indeed the polar dual of a generalized permutohedron since the complete fan $\mathcal{F}= Fan(K)$ is a coarsening of the braid fan.
\end{proof}


\nocite{*}
\bibliographystyle{abbrv}
\bibliography{ref}

\end{document}